\documentclass[reqno,11pt]{amsart}
\usepackage{a4wide,color,eucal,enumerate,mathrsfs}
\usepackage[normalem]{ulem}
\usepackage{amsmath,amssymb,epsfig,amsthm} %,bbm
\usepackage[latin1]{inputenc}
\usepackage{psfrag}

\def\dist{{\mathop\mathrm{\,dist\,}}}
\def\loc{{\mathop\mathrm{\,loc\,}}}

\def\wz{\widetilde}

\def\bint{{\ifinner\rlap{\bf\kern.35em--}
\int\else\rlap{\bf\kern.45em--}\int\fi}\ignorespaces}

\def\bbint{{\ifinner\rlap{\bf\kern.35em--}
\hspace{0.078cm}\int\else\rlap{\bf\kern.45em--}\int\fi}\ignorespaces}

\newtheorem{thm}{Theorem}[section]

\newtheorem{lem}[thm]{Lemma}%[section]     %@@!!@@!!
\newtheorem{ques}[thm]{Question}
\numberwithin{equation}{section}

\theoremstyle{remark}
\def\bint{{\ifinner\rlap{\bf\kern.35em--}
\int\else\rlap{\bf\kern.45em--}\int\fi}\ignorespaces}

\begin{document}

\title[Schoenflies solutions of conformal boundary values may fail to be Sobolev]
{Schoenflies solutions of conformal boundary values may fail to be Sobolev}

\author{Yi Ru-Ya Zhang}

\address{Hausdorff Center for Mathematics, 
Endenicher Allee 60, 
D-53115 Bonn, Germany. 
}
\email{yizhang@math.uni-bonn.de}

\thanks{This work was supported by the Academy of Finland via the Centre of Excellence in
Analysis and Dynamics Research (Grant no. 271983) and the Hausdorff Center for Mathematics. }
\date{\today}

%%%%%%%%%%%%%%%%%%%%%%%%%%%%%%%%%%%%%%%%%%%%%%%%%%%%%%%%%%%%%%%%%%%%%

\begin{abstract}
There exists a planar Jordan domains $\Omega$ with $1$-Hausdorff dimensional boundary such that, for any conformal map $\varphi \colon \mathbb D \to \Omega$, any homeomorphic extensions  to the entire plane of either $\varphi$  or $\varphi^{-1}$ cannot be in $W^{1,\,1}_{\loc}$ class (or even not in $BV_{\loc}$). 
\end{abstract}

%%%%%%%%%%%%%%%%%%%%%%%%%%%%%%%%%%%%%%%%%%%%%%%%%%%%%%%%%%%%%%%%%%%%%

\maketitle

\section{Introduction}

Let $\Gamma\subset  {\mathbb C}$ be a Jordan curve, namely there exists a homeomorphism $\phi\colon \mathbb S^1 \to \Gamma$, where $ {\mathbb C}$ is the complex plane and $\mathbb S^1$ denotes the boundary of the unit disk $\mathbb D$. According to Jordan curve theorem, the curve $\Gamma$ divides $\mathbb C$ into two components, and we call the bounded  component a Jordan domain. 

Jordan-Schoenflies theorem states that any homeomorphism  between two Jordan curves on $ \mathbb C$ can be extended to a homeomorphism between the entire $ \mathbb C$; see \cite[Corollary 2.9]{P1992}. 
To be more precise, given two Jordan domains $\Omega_1$ and $\Omega_2$ and a homeomorphism  $\varphi: \partial \Omega_1\to \partial \Omega_2$, there exists a homeomorphism $\Phi$, which we call a  {\it Schoenflies solution} of the boundary value $\varphi$, from $  {\mathbb C} $ to $  {\mathbb C}$ such that the restriction of $\Phi$ to $\Gamma_1$ coincides with $\varphi$. 
Then a natural question arises: 
\begin{ques}\label{ques}
Given two Jordan domains $\Omega_1,\,\Omega_2\subset  {\mathbb C}$ together with a homeomorphism  $\varphi\colon \partial \Omega_1\to \partial \Omega_2$, what is the best regularity of  Schoenflies solutions  of the boundary value $\varphi$? 
\end{ques}

Certainly the answer to this question depends on the given boundary value and the geometry of both $\Omega_1$ and $\Omega_2$. Let us recall some known results.  If $\Omega_2$ is bounded by a smooth Jordan curve, then by the techniques from differential topology for each conformal map we can find a smooth Schoenflies solution to any homeomorphism from $\mathbb S^1$ onto $\partial \Omega$.  Assume that $\varphi\colon \mathbb S^1 \to \partial \Omega_2$ is quasisymmetric, via Douady-Earle extension theorem there exists a $K$-quasiconformal Schoenflies solution $\Phi$. By \cite{A1994}, we further have that both $\Phi$ and $\Phi^{-1}$ are in $W^{1,\,p}_{\loc}(\mathbb C)$ for any $p<2K/(K-1)$. Recently P. Koskela, P. Pankka and the author have been working on a version of this result for domains satisfying Gehring--Martio conditions \cite{KPZ2017}.

Recall Carath\'eodory's theorem states that, given any two Jordan domain $\Omega_1$ and $\Omega_2$,  every conformal map $\varphi: \Omega_1 \to \Omega_2$ can be continuously extended to the boundary as a homeomorphism $\varphi: \overline{\Omega}_1 \to  \overline{\Omega}_2$. We abuse $\varphi$ here. 
In this paper we investigate Question~\ref{ques} with the boundary value given by Carath\'eodory's theorem, namely a {\it conformal boundary value}.

The main result  of this paper is the following. 

\begin{thm}\label{mainthm 1}
There exists a Jordan domain $\Omega\subset \hat{\mathbb C}$ with $1$-Hausdorff dimensional boundary such that, any Schoenflies solution of  any conformal boundary value $\varphi\colon \mathbb S^1 \to \partial\Omega$ or $\phi \colon \partial\Omega \to \mathbb S^1$ is not in $W^{1,\,1}_{\loc}(\mathbb C)$ (even not in $BV_{\loc}(\mathbb C)$ ). 
\end{thm}

%\begin{thm}\label{mainthm 2}
%There exists a Jordan domain $\Omega\subset \hat{\mathbb C}$ with rectifiable boundary such that,   $\varphi\colon \mathbb D\to \Omega$, any Schoenflies solution of every conformal boundary value $\varphi\colon \partial \Omega \to \mathbb S^1$  is not in $W^{1,\,1}_{\loc}(\mathbb C)$ (even not in $BV_{\loc}(\mathbb C)$ ).
%\end{thm}

%We remark that the two theorems above also holds if we replace the conformality of $\varphi$ by quasiconformality 
%%if such a domain allows extension of quasiconformal mappings to the boundary as a homeomorphism \cite{R1968}
%with the same argument. 
This result indicates that, in general, one cannot expect the regularity of Schoenflies solutions to a given boundary value to be better than homeomorphism; even if the boundary value is given by a (extended) conformal map (which is a quite natural choice). Thus, geometric assumptions on the Jordan domain in question and (energy) controls on the boundary value are necessary. One can see e.g.\ \cite{AIMO2005, V2007, IMS2009, KXW2018} for recent results in this direction. Especially in the very recent paper \cite{KO2018} Koski and Onninen give  positive answers to Question~\ref{ques} under certain circumstances.

The notation in the paper is quite standard. 
%When we make estimates, we sometimes write the constants as 
%positive real numbers $C(\cdot)$ with
%the parenthesis including all the parameters on which the constant depends. The constant $C(\cdot)$ may
%vary between appearances, even within a chain of inequalities.
%By $a\lesssim b$ we mean $a\le C b$ for some constant $C \ge 2$.
The Euclidean distance between two sets $A,\,B \subset \mathbb R^2$ is denoted by $\dist(A,\,B)$. 
We denote by $\ell(\gamma)$ the length of a curve $\gamma$. For a set $A\subset \mathbb R^2$, we write its boundary as $\partial A$ , and its closure as $\overline A$, respectively, with respect to the Euclidean topology.  
We use the notation $\mathcal H^1$ for $1$-dimensional Hausdorff measure. 

{\bf Acknowledgement: } The author would like to express his sincere gratitude to the  referee for his (or her) nice review and useful suggestions on this paper. Especially a serious mistake in the previous version was pointed out.  The author also would like to thank Professor Jani Onninen for posing this interesting question.

\section{Proof of Theorem~\ref{mainthm 1}}

Define the {\it{inner distance with respect to $\Omega$}} between 
$x,\,y\in\Omega$ by
$$\dist_{\Omega}(x,\,y)=\inf_{\gamma\subset \Omega} \ell(\gamma),$$
where the infimum runs over all curves joining $x$ and $y$ in $\Omega.$
%The inner diameter $\diam_{\Omega}(E)$ of a set $E\subset \Omega$ is then defined in
%the usual way.

\subsection{Schoenflies solution of  conformal boundary value $\varphi\colon \mathbb S^1 \to \partial\Omega$}

The idea of the proof is that, we construct a Jordan domain $\Omega\subset \hat{\mathbb C}$ satisfying that there exists a (Cantor) set $E\subset\partial \Omega$ such that,
\begin{enumerate}[(i)]
\item  for any conformal $\varphi\colon\mathbb D \to \Omega$, i.e.\ for any conformal boundary value, we have $$\mathcal H^1(\varphi^{-1}(E))>0;$$
\item for any point $  x$ in the complementary domain $\wz \Omega$, 
$$\dist_{\wz \Omega}(x,\,E\setminus\{(0,\,0),\,(1,\,0)\})=\infty.$$
\end{enumerate}
If such a Jordan domain exists (see Lemma~\ref{harmonic measure} below), then by (i) and (ii), any  Schoenflies solution of the conformal boundary value $\varphi$ is not in $W^{1,\,1}_{\loc}$ (even not in $BV_{\loc}$) by Fubini's theorem;
indeed,   such a solution maps a family of radial segments in the exterior of the unit disk (with finite length) into a family curves of infinite length in $\wz \Omega$. By calculating in the polar coordinate we know that such a map cannot be in   $W^{1,\,1}_{\loc}$ (even not in $BV_{\loc}$). 
Hence Theorem~\ref{mainthm 1} follows.

We first construct a Jordan curve $\Gamma$ in the plane. Towards this,  let us recall the construction of a fat Cantor set $E\subset [0,\,1]$ on the real axis. Let $C_0=I_{0,\,1}=[0,\,1]$ and $C_i$ with $i\ge 1$ recursively as follows: When $I_{i,\,j}=[a,\,b]$ has been defined, let 
$$I_{i+1,\,2j-1}=\left[a,\,\frac {a+b-4^{-i}}{2}\right] \text{ and } \, I_{i+1,\,2j}=\left[\frac {a+b+4^{-i}}{2},\,b\right];$$
i.e.\ we remove an open interval of length $4^{-i}$  from the middle of the interval $I_{i,\,j}$. Then we set
$$C_i=\bigcup_{j=1}^{2^i}I_{i,\,j}.$$
The set $E$ is finally given by
$$E=\bigcap_{i=1}^{\infty}C_i.$$
A simple calculation shows that, for every $i\in \mathbb N$ and $1\le j\le 2^i$, each interval $I_{i,\,j}$ has length 
\begin{equation}\label{dist}
\frac {2^i+1}{2^{2i+1}}\in (2^{-i-1},\,2^{-i}]. 
\end{equation}
Thus $C_i$, and hence $E$ is well-defined. Moreover, $E$ has positive $\mathcal H^1$-measure; note that at  step $i,\,i\ge 1$ there are  $2^i$ intervals removed with total length $2^{-i-1}$. 

We now construct a sequence of simple curves $\gamma_i$ based on the construction of $E$. Again we proceed inductively according to the index $i$. For $i\in \mathbb N$ and $1\le j\le 2^i$, denote by $I'_{i,\,j}\subset I_{i,\,j}$ the interval removed from $I_{i,\,j}$  in the construction of $E$. 
Let $\gamma_0$ be the interval $[0,\,1]$. When $\gamma_{i-1},\,i\ge 1$ has been defined, we replace every open interval $I'_{i,\,j},\,1\le j\le 2^i$, contained in $\gamma_{i-1}$, by a curve 
$$\gamma_{i,\,j}=\partial (I'_{i,\,j}\times [0,\,2^{-i}])\setminus  (I'_{i,\,j}\times \{0\}), $$
consisting of three line segments, where $\times$ means the Cartesian product.  We then obtain $\gamma_i$. See Figure~\ref{1}.
\begin{figure}
 \centering
 \def\svgwidth{300pt}
 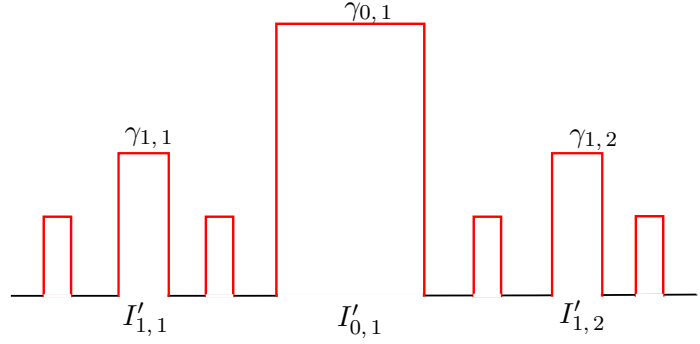
 \caption{The curve $\gamma_2$ is shown in the figure. In the previous steps the intervals $I'_{0,\,1},\,I'_{1,\,1},\,I'_{1,\,2}$ were replaced by curves $\gamma_{0,\,1},\,\gamma_{1,\,1},\,\gamma_{1,\,2}$, and in the current step four more intervals are replaced.}
 \label{1}
\end{figure}
Since $\{\gamma_i\}$ (under suitable parameterizations) is a Cauchy sequence of curves in the plane with respect to the supremum distance, the limit $\gamma$ exits and  is a curve. Moreover, $\gamma$ is simple. 

For fixed $n\in \mathbb N$, there are $2^{n+1}-1$ curves $\gamma_{i,\,j}$ intersecting $\mathbb R\times (2^{-n-1},\,2^{-n}]$. Indeed, if $\gamma_{i,\,j} \cap (\mathbb R\times (2^{-n-1},\,2^{-n}]) \neq \emptyset$, then $i\le n$.  The distance between any two of these curves is strictly larger than $2^{-n-1}$ by \eqref{dist}.

We next construct a sequence of new curves $\Gamma_n$ according to the index $n$. 
First of all define $\Gamma_0=\gamma.$ When $\Gamma_{n-1},\,n\ge 1$ has been defined, we modify the segments in 
$$\gamma_{i,\,j} \cap (\mathbb R\times (2^{-n-1},\,2^{-n}]),\ 0\le i\le n$$
 to obtain $\Gamma_n$. 
Recall that $\gamma_{i,\,j}$ replaces the interval $I'_{i,\,j}$ in the construction of $\gamma_i$. Denote by $a_{i,\,j}$ and $b_{i,\,j}$ the end points of $I'_{i,\,j}$ with $a_{i,\,j}<b_{i,\,j}$. 
Then for every $1\le i\le n$, 
$$\gamma_{i,\,j}\cap (\mathbb R\times(2^{-n-1},\,2^{-n}])=(\{a_{i,\,j}\}\times (2^{-n-1},\,2^{-n}]) \cup (\{b_{i,\,j}\}\times (2^{-n-1},\,2^{-n}])$$
and each $1\le k \le 2^{n}-1$, we replace each segment 
$$\{a_{i,\,j}\}\times [2^{-n-1}+(4k) 2^{-2n-3},\,2^{-n-1}+ (4k+1) 2^{-2n-3}]$$
by 
\begin{align*}
A_{i,\,j}^{n,\,k}:=&\partial \left([a_{i,\,j}-2^{-n-1},\,a_{i,\,j}]\times [2^{-n-1}+(4k) 2^{-2n-3},\,2^{-n-1}+ (4k+1) 2^{-2n-3}]\right) \\ 
& \qquad \setminus \{a_{i,\,j}\}\times [2^{-n-1}+(4k) 2^{-2n-3},\,2^{-n-1}+ (4k+1) 2^{-2n-3}], 
\end{align*}
and 
$$\{b_{i,\,j}\}\times [2^{-n-1}+(4k+2) 2^{-2n-3},\,2^{-n-1}+ (4k+3) 2^{-2n-3}]$$
by 
\begin{align*}
B_{i,\,j}^{n,\,k}:=&\partial \left([b_{i,\,j},\,b_{i,\,j}+2^{-n-1}]\times [2^{-n-1}+(4k+2) 2^{-2n-3},\,2^{-n-1}+ (4k+3) 2^{-2n-3}]\right) \\ 
& \qquad \setminus \{b_{i,\,j}\}\times [2^{-n-1}+(4k+2) 2^{-2n-3},\,2^{-n-1}+ (4k+3) 2^{-2n-3}]. 
\end{align*}
This gives us the new curve $\Gamma_n$.  See Figure~\ref{2}. 

\begin{figure}
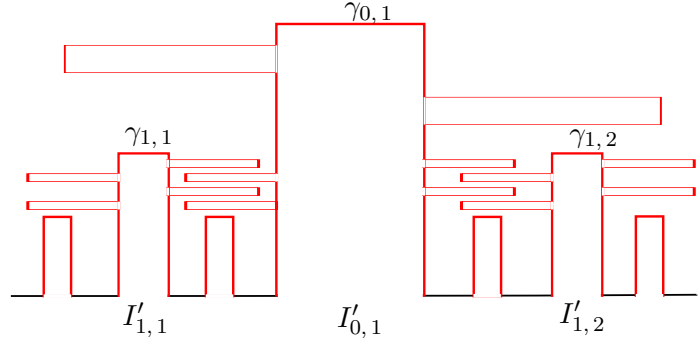

 \centering
 \def\svgwidth{300pt}
 
 \caption{The curve $\Gamma_2$ is shown in the figure, with the replacement of certain segments contained in $\Gamma_1$ by parts of boundaries of some rectangles, receptively.}
 \label{2}
\end{figure}

Again since $\{\Gamma_n\}$ (under suitable parameterizations) is a Cauchy sequence of curves in the plane with respect to the supremum distance, we conclude that $\Gamma_n$ converges uniformly to some curve $\Gamma_{\infty}$ as $n\to \infty$. Moreover, according to our construction $\Gamma_{\infty}$ is simple. 
Define 
$$\Gamma=\Gamma_{\infty}\cup \left(\partial  ([0,\,1]\times [-1,\,0] )\setminus [0,\,1]\times \{0\}\right).$$
Since $\Gamma_{\infty}$ is simple, then also $\Gamma$ is simple, and hence Jordan as it is closed. 
We denote by $\Omega$ the bounded component of $\mathbb C\setminus \Gamma. $ 

Notice that $\partial \Omega$ is a countable union of rectifiable curves, even though it does not have finite length. Since the Hausdorff dimension of a countable union of sets is the supremum of the Hausdorff dimensions of the sets, see e.g.\ \cite[Page 81, Section 5.9]{M1995}, we conclude that $\partial \Omega$ is a set of Hausdorff dimension $1$.
\vskip 0.15 cm\noindent

Recall the Cantor set $E$. Now let us check that the Jordan domain $\Omega$ satisfies the two properties (i) and (ii). We remark that, in our construction, for any point $x$ in $\Omega$, we have
$$\dist_{\Omega}(x,\,E)<\infty.$$

Before showing (i), we note that, property (i) is stated with respect to all conformal maps. However, since two such Riemann maps differ from each other by a M\"obius transform on the unit disk,
we may assume that $\varphi(0)$ is the center of the square $[0,\,1]\times [-1,\,0]$.

Recall that the harmonic measure in the unit disk is defined via the Poisson kernel, and then in any Jordan domain via the (extended) Riemann mapping. 
For $E\subset \partial \Omega$, we use $\omega(x_0,\,E,\,\Omega)$ to designate {\it the harmonic measure of $E$ at $x_0$ in $\Omega$}. 
It is known that 
 $\omega(x_0,\,E,\,\Omega)= u(x_0)$
where $u$ is the (unique) harmonic function in $\Omega$ whose boundary value is the characteristic function of $E$ on $\partial \Omega$. 
We refer to \cite{GM2005} for more details. 

\begin{lem}\label{harmonic measure}
The Jordan domain $\Omega$ constructed above satisfies properties (i) and (ii).  
\end{lem}
\begin{proof}
Towards (i), we first observe that 
\begin{equation}\label{comparison}
\omega(\varphi(0),\,E,\,\Omega)\ge \omega(\varphi(0),\,E,\,Q)>0,
\end{equation}
 where $Q$ is the open square $(0,\,1)\times (-1,\,0)$. Indeed, the first estimate comes from the comparison principle of harmonic measures, %Indeed, let $u_1$ and $u_2$ be the harmonic functions on $\Omega$ and $Q\subset \Omega$, respectively, with  the characteristic function of $E$ on the corresponding boundaries as the boundary values. 
%By zero-extension to $\Omega$ we have that $u_2$ is a subharmonic function on $\Omega$ with the same boundary value as $u_1$. Thus 
% $$u_1(\varphi(0))\ge u_2(\varphi(0))>0,$$
while the second inequality follows from F.\  and M.\ Riesz theorem since its $1$-Hausdorff measure is strickly positive. 
 
By the conformal invariance of harmonic measure we have 
$$\omega(0,\,\varphi^{-1}(E),\,\mathbb D)>0.$$
According to the definition of harmonic measures in the unit disk, we conclude (i). 

To show (ii), note that in our construction, any curve in the unbounded component of $\mathbb R^2\setminus \Gamma$ towards $E\setminus\{(0,\,0),\,(1,\,0)\}$ has length at least $\frac 1 2$ in the region $\mathbb R\times (2^{-n-1},\,2^{-n}] $ for $n$ large enough; the curve has to oscillate $2^n$ times and each time it goes at least $2^{-n-1}$. This implies that any curve in the unbounded component of $\mathbb R^2\setminus \Gamma$ towards $E\setminus\{(0,\,0),\,(1,\,0)\}$ has infinite length. Property (ii) is complete. 
\end{proof}

\subsection{Schoenflies solution of  conformal boundary value $\phi\colon \partial\Omega \to \mathbb S^1$}
 Let $\phi\colon \Omega \to \mathbb D$ be a conformal map giving the conformal boundary value via Carath\'eodory's theorem. 
By composing with a suitable M\"obius map, we may assume that $\phi(z_0)=0$, where $z_0$ is the center of open square $Q=(0,\,1)\times (-1,\,0)$; in the general case the constants below will further depending on the   M\"obius transform. 
We show that any homeomorphic extension of $\phi$ is not in $W^{1,\,1}_{\loc}$. 

Towards this, recall that in the construction of $\Gamma=\partial \Omega$ we attached ``arms'' $A_{i,\,j}^{n,\,k}$ and $B_{i,\,j}^{n,\,k}$ to every curve $\gamma_{i,\,j}$. We first claim that, there exists an absolute constant $c>0$ such that, for $n\ge 3$,
\begin{equation}\label{dist in disk}
\dist(\phi(A_{i,\,j}^{n,\,k}),\,\phi(B_{i',\,j'}^{n,\,k'}))\ge c 2^{-n} 
\end{equation}
whenever $I'_{i,\,j},\,I'_{i',\,j'} \subset \left[\frac 5 {32},\,\frac {27} {32}\right]$ and either $i\neq i'$ or $j\neq j'$. 

Indeed, let us fix $A_{i,\,j}^{n,\,k}$ and $B_{i',\,j'}^{n,\,k'}$. 
According to our construction, there exists an interval $J\in \left\{I_{n+1,\,j}\right\}_{j=1}^{2^{n+1}}$ such that $J\subset \left[\frac 5 {32},\,\frac {27} {32}\right]$ is between $I'_{i,\,j}$ and $I'_{i',\,j'}$.  Since $\phi\colon \partial\Omega \to \mathbb S^1$ is a homeomorphism, by the construction of $\partial\Omega$ and the geometry of the unit circle, we have that
$$\dist(\phi(A_{i,\,j}^{n,\,k}),\,\phi(B_{i',\,j'}^{n,\,k'}))\ge c_1\mathcal H^1(\phi(J\cap E))$$
for some absolute constant $c_1$. 
Therefore it suffices to show that $\mathcal H^1(\phi(J\cap E))\ge c_2 2^{-n}$ for some absolute constant $c_2$. 

Again by the invariance of harmonic measure under conformal map and the comparison principle of harmonic measures, 
$$\omega(0,\, \phi(J\cap E),\,\mathbb D)=\omega(z_0,\,J\cap E,\,\Omega)\ge \omega(z_0,\,J\cap E,\,Q).$$
According to Schwarz-Christoffel formula \cite[Chapter 3.1]{P1992}, since $J\subset  \left[\frac 5 {32},\,\frac {27} {32}\right]$ is away from the corner of $Q$, we have
$$\omega(z_0,\,J,\,Q)\ge c_3 2^{-n}$$
for some absolute constant $c_3$; note that the length of $J$ is $2^{-n-2}+2^{-1}4^{-n-1}$, and $E$ is a self-similar fat Cantor set. Therefore, we conclude \eqref{dist in disk} via the Poisson formula in the unit disk. 

Let $\Phi$ be any  Schoenflies solution of the conformal boundary value $\phi$. By \eqref{dist in disk}, the image under $\phi$ of any vertical segment joining ``neighboring arms'' $A_{i,\,j}^{n,\,k}$ and $B_{i',\,j'}^{n,\,k}$ in the exterior of $\Omega$
 has length at least $c 2^{-n}$. 
Moreover, when $n\ge 3$, the intersection of the projections on the real axis of the  ``neighboring arms'' $A_{i,\,j}^{n,\,k}$ and $B_{i',\,j'}^{n,\,k}$ is an interval with length not less than $2^{-n-2}$, and there are at least $4^n$ pairs of those  ``neighboring arms'' contained in $\left[\frac 5 {32},\,\frac {27} {32}\right]\times [2^{-n-1},\,2^{-n}]$ up to a multiplicative constant. 
 Therefore by Fubini's theorem we conclude
\begin{align*}
\int_{\left[\frac 5 {32},\,\frac {27} {32}\right]\times [2^{-n-1},\,2^{-n}]} |D\Phi|\, dx\ge c' 2^{-n} 2^{-n-2} 4^n \ge  {c'} 2^{-2}
\end{align*}
for some absolute constant $c'$.
Therefore, in any Euclidean neighborhood of $E\cap \left[\frac 5 {32},\,\frac {27} {32}\right]$ the $W^{1,\,1}_{\loc}$-energy of $\Phi$ is infinite, and a similar argument shows that $\Phi\notin BV_{\loc}$. This concludes the second part of Theorem~\ref{mainthm 1}.

\end{document}